\documentclass[reqno,twoside,11pt]{amsart}

\usepackage{cite}
\usepackage{amsmath}
\usepackage{amsfonts}
\usepackage{amssymb}
\usepackage{verbatim}
\usepackage{epsfig}
\usepackage{color}

\definecolor{Blue}{rgb}{0,0,1}

\DeclareFontFamily{OT1}{eusb}{} \DeclareFontShape{OT1}{eusb}{m}{n} {<5> <6> <7> <8> <9> <10> <11> <12> <14.4> eusb10}{}
\DeclareMathAlphabet{\eusb}{OT1}{eusb}{m}{n}

\DeclareFontFamily{OT1}{eusm}{} \DeclareFontShape{OT1}{eusm}{m}{n} {<5> <6> <7> <8> <9> <10> <11> <12> <14.4> eusm10}{}
\DeclareMathAlphabet{\eusm}{OT1}{eusm}{m}{n}

\DeclareFontFamily{OT1}{eufm}{} \DeclareFontShape{OT1}{eufm}{m}{n} {<5> <6> <7> <8> <9> <10> <11> <12> <14.4> eufm10}{}
\DeclareMathAlphabet{\mathfrak}{OT1}{eufm}{m}{n}

\DeclareFontFamily{OT1}{fraktura}{}
\DeclareFontShape{OT1}{fraktura}{m}{n} {<5> <6> <7> <8> <9> <10> <11> <12> <13> <14.4> [1.1] eufm10}{}
\DeclareMathAlphabet{\fraktura}{OT1}{fraktura}{m}{n}

\DeclareFontFamily{OT1}{cmfi}{} \DeclareFontShape{OT1}{cmfi}{m}{n} {<5> <6> <7> <8> <9> <10> <11> <12> <13> <14.4> [0.9] cmfi10}{}
\DeclareMathAlphabet{\cmfi}{OT1}{cmfi}{b}{n}

\DeclareFontFamily{OT1}{cmss}{} \DeclareFontShape{OT1}{cmss}{m}{n} {<5> <6> <7> <8> <9> <10> <11> <12> <13> <14.4> cmss10}{}
\DeclareMathAlphabet{\cmss}{OT1}{cmss}{m}{n}

\newcommand{\cut}{{\mathcal C}}

\newtheoremstyle{thm}{1.5ex}{1.5ex}{\itshape\rmfamily}{} {\bfseries\rmfamily}{}{2ex}{}

\newtheoremstyle{def}{1.5ex}{1.5ex}{\slshape\rmfamily}{} {\bfseries\rmfamily}{}{2ex}{}

\newtheoremstyle{rem}{1.3ex}{1.3ex}{\rmfamily}{} {\itshape}
{} {1.5ex}{}

\theoremstyle{thm}
\newtheorem{theorem}{Theorem}[section]
\newtheorem{lemma}[theorem]{Lemma}
\newtheorem{claim}[theorem]{Claim}
\newtheorem{proposition}[theorem]{Proposition}

\newtheorem*{Main Theorem}{Main Theorem.}
\newtheorem{corollary}[theorem]{Corollary}
\newtheorem*{special theorem}{Lindeberg-Feller Theorem for Martingales}

\theoremstyle{def}
\newtheorem{definition}[theorem]{Definition}

\theoremstyle{rem}

\numberwithin{equation}{section}

\renewcommand{\section}{\secdef\sct\sect}
\newcommand{\sct}[2][default]{%
\refstepcounter{section}
\addcontentsline{toc}{section}{{\tocsection {}{\thesection}{\!\!\!\!#1\dotfill}}{}}
\vspace{0.7cm}
\centerline{\scshape\thesection.\ #1} \nopagebreak \vspace{0.2cm}}
\newcommand{\sect}[1]{%
\vspace{0.4cm} \centerline{\large\scshape\rmfamily #1}
\vspace{0.2cm}}
\newcommand{\tp}{{t^\prime}}
\newcommand{\qx}{{Q_{\XX}}}

\newcommand{\var}{{\bf var}}
\newcommand{\vo}{{\eta}}
\newcommand{\hfC}{{\mathcal D}}
\newcommand{\ignore}[1]{{}}

\renewcommand{\subsection}{\secdef\subsct\sbsect}
\newcommand{\subsct}[2][default]{\refstepcounter{subsection}
\addcontentsline{toc}{subsection}
{{\tocsection{\!\!}{\hspace{1.2em}\thesubsection}{\!\!\!\!#1\dotfill}}{}}
\nopagebreak\vspace{0.45\baselineskip} {\flushleft\bf
\thesubsection~\bf #1.~}
\\*[3mm]\noindent
\nopagebreak}
\newcommand{\sbsect}[1]{\vspace{0.1cm}\noindent
\textbf{#1.~}\vspace{0.1cm}}

\renewcommand{\subsubsection}{%
\secdef \subsubsect\sbsbsect}
\newcommand{\subsubsect}[2][default]{%
\refstepcounter{subsubsection}
\addcontentsline{toc}{subsubsection}{{\tocsection{\!\!}
{\hspace{3.05em}\thesubsubsection}{\!\!\!\!#1\dotfill}}{}}
\nopagebreak
\vspace{0.15\baselineskip} \nopagebreak {\flushleft\rmfamily
\itshape\thesubsubsection
\ \rmfamily #1\/.}\ }
\newcommand{\sbsbsect}[1]{\vspace{0.1cm}\noindent
\rmfamily \itshape
\arabic{section}.\arabic{subsection}.\arabic{subsubsection} \
\sffamily #1\/.\ }

\renewcommand{\caption}[1]{%
\vglue0.5cm
\refstepcounter{figure}
\begin{minipage}{0.9\textwidth}\small {\sc Figure~\thefigure. }#1\end{minipage}}

\newcommand{\omegan}{\omega^{(n)}}

\newcommand{\tq}{\widetilde{Q}}

\newcommand{\prob}{{\bf P}}

\newcommand{\PP}{\Gamma}

\newcommand{\XX}{\mathcal X}

\newcommand{\C}{\mathbf C}

\newcommand{\E}{{\bf E}}

\newcommand{\N}{\mathbb N}

\newcommand{\R}{\mathbb R}

\newcommand{\Z}{\mathbb Z}

\newcommand{\tP} {{\widetilde{P}}}
\newcommand{\tQ} {{\widetilde{Q}}}

\newcommand{\br}{{\rm br}}

\def\myffrac#1#2 in #3{\raise 2.6pt\hbox{$#3 #1$}\mkern-1.5mu\raise 0.8pt\hbox{$#3/$}\mkern-1.1mu\lower 1.5pt\hbox{$#3 #2$}}

\title[Detection of RW trails]
{\fontsize{14}{20}\selectfont Detecting the trail of a random walker in a random scenery}
\author[N.~Berger and Y.~Peres]{Noam Berger\,$^1$ \and Yuval Peres\,$^2$}

\begin{document}
\maketitle

\vspace{-5mm}
\centerline{\textit{$^1$The Hebrew University of Jerusalem and the Technical University of Munich}}
\centerline{\textit{$^2$Microsoft Research, Redmond}}

\begin{abstract}
Suppose that the vertices of the lattice $\Z^d$ are endowed with a random scenery, obtained
 by tossing a fair coin at each vertex. A random walker, starting from the origin,
 replaces the coins along its path by i.i.d.\ biased coins.
For which walks and dimensions can the resulting scenery be distinguished
 from the original scenery?
We find the answer for simple random walk, where it does not depend on dimension, and for
walks with a nonzero mean, where a transition occurs between dimensions three and four.
We also answer this question for other types of graphs and walks, and raise several new questions.
\end{abstract}

\vspace{0.5cm}


\section{Introduction}
\label{sec:intro}\noindent

Let $\mu$ and $\nu$ be different probability measures on the same
finite sample space $\Omega$, such that $\mu(\omega)>0$ and $\nu(\omega)>0$
 for all $\omega \in \Omega$. Let $G$ be an infinite graph with a distinguished vertex
$v$, and denote by   $\PP(G,v)$ the space of infinite
paths $v,v_1,v_2,\ldots$ in $G$ that emanate from $v$. Endow $\PP(G,v)$
with the induced product topology, and let
$\Psi$ be a Borel probability measure on $\PP(G,v)$.
Suppose that initially, i.i.d. labels with law $\mu$ are attached to the vertices of $G$.
(Call the law of this random scenery $P$).
In the perturbation step, an infinite random path $\XX$ with distribution $\Psi$ is chosen,
and the labels along $\XX$ are replaced by independent labels with law $\nu$; this yields a new random scenery
with distribution $Q$. We address the
{\bf perturbation detection problem:}  \it
Given a scenery in $\Omega^{V(G)}$, can one distinguish (without knowing $\XX$) whether
this scenery was sampled from $P$ or from $Q$? \rm

Clearly, the answer depends on the choice of $G,\Psi,\mu$ and $\nu$;
as we shall see, it is sometimes quite surprising.
To state this problem formally,
 let $P$ be the product measure $\mu^{V(G)}$, which is the initial distribution of the scenery.
 The distribution $Q$ of the perturbed scenery
is constructed as follows.
Denote by  $[\XX]$ the set of vertices in $\XX \in \PP(G,v)$ and
let $Q_\XX$ be the product measure $Q_\XX=\mu^{V(G)-[\XX]}\times\nu^{[\XX]}$
on  $\Omega^{V(G)}$ (i.e., the labels off $[\XX]$
are sampled from $\mu$ and the labels on $[\XX]$ are sampled from $\nu$.)
Finally, define the Borel measure $Q$ on $\Omega^{V(G)}$ by
\[
Q(A)=\int_{\PP(G,v)} Q_\XX(A) \, d\Psi(\XX) \,.
\]

We say that the  distributions $P$ and $Q$ are {\bf indistinguishable} if $P$
and $Q$ are absolutely continuous with respect to each other; otherwise, we say that
$P$ and $Q$ are {\bf distinguishable}. In general,  examples exist
of measures $P,Q$, constructed as above, that are distinguishable but not singular. However,
throughout most of this paper we choose to focus on graphs $G$ and
path distributions $\Psi$ where this intermediate situation does not occur.
Indeed, this can be established when
 $\Psi$ is the law of an automorphism-invariant Markov chain on a transitive
  graph.


\begin{proposition}\label{prop:01law}
Let $G$ be a transitive graph and let $M$ be a
transition kernel on $V(G)$ which is invariant under a transitive subgroup
$H$ of automorphisms of $G$. (That is, $M(h(x),h(y))=M(x,y)$ for $x,y \in V(G)$
and $h \in H$.)
Let $\Psi$ be the law of the Markov chain with transition law $M$ and initial state $v$;
we assume that this chain is transient.
 Then the measures $P$ and $Q$ are either
singular, or mutually absolutely continuous.
\end{proposition}

In the subsections below we present the main results of this
paper, which determine whether  $P$ and
$Q$ are distinguishable for several families of graphs and random paths.

\subsection{The Euclidean lattice}
We contrast the behavior of simple random walk on $\Z^d$ with a walk of nonzero mean.
\begin{theorem}\label{thm:zdcase}
Assume that $G$ is the Euclidean lattice $\Z^d$.
\begin{enumerate}
\item\label{item:meanz}
Let $\Psi$ be the law of simple random walk on $\Z^d$.
Then for all dimensions $d$ and all
$\mu\neq\nu$, the distributions $P$ and $Q$ are singular.
\item\label{item:orient}
Let $\Psi$ be the law of a nearest-neighbor random walk on $\Z^d$,
 with i.i.d.\ increments of nonzero mean.
If $d\leq 3$, then for all $\mu\neq\nu$, the distributions $P$ and $Q$ are singular; however, if
$\, d\geq 4$, then there exist  $\mu\neq\nu$ such that the distributions
$P$ and $Q$ are indistinguishable.
\item\label{item:genpath3}
For $d\geq 3$, there exists a (not necessarily Markovian)
distribution $\Psi$ on $\PP(\Z^d,0)$ and measures $\mu\neq\nu$, such that $P$ and $Q$ are
indistinguishable.
\item\label{item:genpath2}
For any distribution
$\Psi$ on $\PP(\Z^2,0)$ and every $\mu\neq\nu$, the measures $P$ and $Q$
are singular.
\end{enumerate}
\end{theorem}

\noindent{\bf Remark.} Part (\ref{item:orient}) of Theorem \ref{thm:zdcase}
is closely related to a result of Bolthausen and Sznitman \cite{BS}
on random walks in random environment.

\subsection{General graphs}
In this subsection we focus on
simple random walk, and prove the following fact.
(For definitions of speed of random walks and nonamenable graphs see, e.g., \cite{lyons-peres}.)
\begin{theorem}\label{thm:gengrp} 
\begin{enumerate}
\item\label{item:pos_speed}
Let $G$ be a Cayley graph such that simple random walk on $G$ has
positive speed. Then there exist $\mu\neq\nu$ such that $P$ and
$Q$ are indistinguishable.
\item\label{item:non_am}
Let $G$ be transitive and nonamenable.  Then there exist $\mu\neq\nu$ such
that $P$ and $Q$ are indistinguishable and the Radon-Nikodym
derivative $\frac{dQ}{dP}$ is in $L^2(P)$.
\end{enumerate}
\end{theorem}

\subsection{Self-avoiding walks on trees --- a relative entropy criterion}

The next case that we discuss is self-avoiding walks on trees. In
this case we find a distinguishability criterion
in terms of the relative entropy between $\mu$ and $\nu$. Since we
discuss general trees (which are typically not transitive),
Proposition \ref{prop:01law} no longer applies, so $P$ and $Q$
 might be neither absolutely continuous nor singular with respect
to each other. Before we state the theorems, we need a few
definitions.

\begin{definition}\label{def:relent}
For measures $\mu$ and $\nu$ on $\Omega$, the {\em entropy} of
$\nu$ {\em relative} to $\mu$ is defined as
\[
H(\nu|\mu)=\sum_{\rho\in\Omega}\nu(\rho)\log\left(\frac{\nu(\rho)}{\mu(\rho)}\right).
\]
\end{definition}
(Recall that we always assume that $\mu(\omega)>0$ and $\nu(\omega)>0$
 for all $\omega \in \Omega$.)
Consider an infinite tree without leaves (except possibly at the
root). A {\em ray} is an infinite self-avoiding path starting at
the root. The {\em boundary} $\partial T$ of a tree $T$ is the set of all rays;
thus $\partial T \subset \PP(T,\mbox{root})$.
The induced Borel $\sigma$-algebra on $\partial T$ is generated by the sets
$\{\Upsilon(v):v\in T\}$, where $\Upsilon(v) \subset \partial T$ denotes the set of all rays going
through the vertex $v$. For a Borel measure $\Psi$ on  $\partial T$, we
abbreviate $\Psi(v)$ for $\Psi(\Upsilon(v))$.

\medskip

\begin{definition}\label{def:bn} \! \! \! \! \! \! \mbox{ (\rm Lyons \cite{russell})} \! \!
The {\bf branching number} $\br(T)$ of a tree $T$ is defined as the
supremum of all values $\beta$ such that there exists a
probability measure $\Psi_\beta$ on $\partial T$
satisfying
\[
\sup_{v\in V(T)}\beta^{|v|}\Psi_\beta(v)<\infty,
\]
where $|v|$ denotes the distance between $v$ and the root.

\end{definition}

\begin{definition}\label{def:locdim}
Let $\Psi$ be a probability measure on $\partial T$ and let
$\chi=(v,v_1,v_2,\ldots)$ be a ray in $\partial T$. The {\em local
dimension} of $\Psi$ on $\chi$ is
\[
d_{\Psi}(\chi)=\liminf_{n\to\infty}\frac{-\log(\Psi(v_n))}{n}.
\]
\end{definition}

Let $T$ be a leafless tree and let $\mu\neq\nu$ be probability measures
supported on a finite space $\Omega$.
As before, for a distribution $\Psi$ on $\partial T$,
the probability measures $P$ and $Q$ on
$\Omega^{V(T)}$ are defined by:
\[
P=\mu^{V(T)}\ \ \ ; \ \ \ Q=Q_\Psi=\int_{\partial
T}\nu^{[\XX]}\times\mu^{V(T)-[\XX]}d\Psi(\XX).
\]
\begin{theorem}\label{thm:branum} With notation as above,
\begin{enumerate}
\item\label{item:sub}
If $\log \br(T)<H(\nu|\mu)$, then $P$ and $Q$ are singular for every
measure $\Psi$ on $\partial T$.
\item\label{item:super}
If $\log \br(T)>H(\nu|\mu)$, then there exists a measure $\Psi$ on
$\partial T$ such that $P$ and $Q$ are indistinguishable.
\end{enumerate}
\end{theorem}

\begin{theorem}\label{thm:sinpath}
Let $T$ be a leafless tree and let $\Psi$ be a measure on $\partial T$.
Let $\Upsilon_+$ be the set of rays $\chi \in \partial T$ such that
 $d_{\Psi}(\chi)>H(\nu|\mu)$; similarly, let  $\Upsilon_-$ be the set of
 $\chi \in \partial T$ such that
 $d_{\Psi}(\chi)<H(\nu|\mu)$. Denote $\Psi$ conditioned on $\Upsilon_+$ by $\Psi_+$, and define
 $\Psi_-$ analogously. We write $Q_+$ for $Q_{\Psi_+}$ and $Q_-$ for $Q_{\Psi_-}$.
\begin{enumerate}
\item\label{item:sinsub}
if $\Psi(\Upsilon_+)>0$, then $Q_+$ and $P$ are indistinguishable.
\item\label{item:sinsuper}
If $\Psi(\Upsilon_-)>0$, then $Q_-$ and $P$ are singular.
\end{enumerate}
\end{theorem}

\subsection{Structure of the paper}
In Section \ref{sec:euc} we prove Theorem \ref{thm:zdcase}. In
Section \ref{sec:gengrp} we prove Proposition \ref{prop:01law}
 and Theorem \ref{thm:gengrp}, and in
Section \ref{sec:tree} we prove Theorem \ref{thm:branum} and
Theorem \ref{thm:sinpath}.

\noindent{\bf Remark.} After the results of this paper were obtained, we learned that
Arias-Castro, Candes, Helgason and Zeitouni~\cite{Ofer} considered some related questions.
However, the emphasis in their paper is different, and the overlap
between the two papers is minimal.

\section{Distinguishability in the Euclidean lattice}\label{sec:euc}
\subsection{Mean zero Random Walks}
In this subsection we prove part \ref{item:meanz} of Theorem
\ref{thm:zdcase}. To establish singularity between $P$ and $Q$,
we use the fact that typically, there exist cubes of volume significantly greater than
$\log t$, such that simple
random walk visits a substantial portion of the cube in the first $t$ steps.
\begin{lemma}\label{lem:largediv}
For $d\geq 3$, let $\{\XX_j\}$ be a  mean zero random walk on
$\Z^d$, let $T_k=\min\{j:\|\XX_j\|_\infty=k\}$ and let $[\XX^{(k)}]$
be the set of points covered by $\{\XX_j:j=1,\ldots,T_k\}$. There
exist $\delta >0$ and $C <\infty$, such that
\[
\prob\left( \left| [-n,n)^d\cap\left[\XX^{(2n)}\right] \right| \geq
\delta (2n)^d \right) \ge \delta  e^{-Cn^{d-2}}.
\]
for every sufficiently large $n$.
\end{lemma}

\begin{proof}
Throughout this proof, all norms are $\ell^\infty$ norms. Fix $n$ and define
the stopping time $S_1=\min\{j:\|\XX_j\|=n\}$. Let
$R_1=\min\{j\in(S_1,T_{2n}):\|\XX_j\|=n/2\}$, with the convention that
$R_1$ is $\infty$ if the set is empty. For $k=2,3,\ldots$
satisfying $R_{k-1}<\infty$, we define
$S_k=\min\{j>R_{k-1}:\|\XX_j\|=n\}$ and
$R_k=\min\{j\in(S_k,T_{2n}):\|\XX_j\|=n/2\}$ where, again,
$R_k=\infty$ if the set is empty. By Donsker's invariance
principle, there exists $C<\infty$ (that does not depend on $n$) such that
\[
\prob(R_k<\infty\ |\ R_{k-1}<\infty\ ;\
\XX_1,\XX_2,\ldots,\XX_{R_{k-1}})\ge e^{-C}.
\]
Let $A$ be the event $\{R_{n^{d-2}}<\infty\}$. Then
$\prob(A)\geq e^{-Cn^{d-2}}$. Consider the cubical shell $W=\{x:2n/3<\|x\|\leq 5n/6\}$. By
Green function estimates (see e.g. \cite{lawler}), there exists $c_1>0$ such that
\[
\prob
 \left(
 \exists{t\in(R_k,S_{k+1}]}: \XX_t=x \left| \{\XX_i\}_{i=1}^{R_k}\right.
 \right)\geq c_1n^{2-d}\cdot{\bf 1}_{R_k<\infty} \ \ \ \mbox{a.s.}
\]
for every $k$ and every $x\in W$.
 Therefore
$\prob(x\in\left[\XX^{(2n)}\right]|A)\geq \rho>0$ for every $x\in
W$. Consequently,
\[
\E\left(\left.\left|[-n,n)^d\cap\left[\XX^{(2n)}\right]\right|\
\right|\ A\right) \ge c_2 (2n)^d.
\]
for some constant $c_2=c_2(d)$. But
$\left|[-n,n)^d\cap\left[\XX^{(2n)}\right]\right|\leq (2n)^d$,
whence $\delta=c_2/2$ satisfies
\[
\prob\left(\left.\left|[-n,n)^d\cap\left[\XX^{(2n)}\right]\right| \ge \delta
(2n)^d\ \right|\ A\right) \ge \delta,
\]
so
\[
\prob\left(\left|[-n,n)^d\cap\left[\XX^{(2n)}\right]\right|\ge \delta
(2n)^d \right)\ge \delta \,\prob(A) \ge \delta e^{-Cn^{d-2}}.
\]
\end{proof}

\begin{proof}[Proof of Part \ref{item:meanz} of Theorem \ref{thm:zdcase}]
Since $\mu\neq\nu$, there  exists some $\rho\in\Omega$ with
$\nu(\rho)>\mu(\rho)$. Let $k(n)=(\log n)^\alpha$ with
$\alpha=1/(d-1)$. The singularity of $P$ and $Q$ follows
from the following claim.

\begin{claim}\label{claim:insquare}

\begin{enumerate}
Let $\delta$ be as in Lemma \ref{lem:largediv}. For every $n$, let
$A_n$ be the event that there exists a cube $\Lambda$ of side
length $k(n)$ in $[-n,n)^d$ such that
\begin{equation}\label{eq:incube}
\frac{\left|\left\{x\in
\Lambda:\omega(x)=\rho\right\}\right|}{|\Lambda|}
>\mu(\rho)+\frac{\delta\left(\nu(\rho)-\mu(\rho)\right)}{2}.
\end{equation}
\item\label{item:pnosquare}
$P$-almost surely, $A_n$ occurs only for finitely many values of $n$.
\item\label{item:qsquare}
$Q$-almost surely, $A_n$ occurs for all large enough $n$.
\end{enumerate}
\end{claim}
\end{proof}
\begin{proof}[Proof of claim \ref{claim:insquare}]
Part (\ref{item:pnosquare}) follows immediately from standard large
deviation bounds:
 For every cube $\Lambda$ of side length $k(n)$,
 \[
 \prob\left(
\frac{\left|\left\{x\in
\Lambda:\omega(x)=\rho\right\}\right|}{|\Lambda|}
>\mu(\rho)+\frac{\delta\left(\nu(\rho)-\mu(\rho)\right)}{2}\
 \right)\le e^{-c|\Lambda|}=e^{-c(\log n)^\frac{d}{d-1}},
 \]
 for some  $c=c(\delta,\mu,\nu)>0$.
 Since there are at most $2^d n^d$ such cubes $\Lambda$ in $[-n,n)^d$,
\[
 \prob(A_n)\le 2^dn^de^{-c(\log n)^\frac{d}{d-1}} \,.
\]
 Thus $\sum_{n=1}^\infty\prob(A_n)<\infty$, so by Borel-Cantelli
 only finitely many of the events $A_n$ occur.

 Part (\ref{item:qsquare}) can be deduced from Lemma
\ref{lem:largediv}. Indeed, fix $n$ and let $\{\XX_j\}$
be the random walk. for $\ell=1,\ldots,\sqrt{n}$, let
$j(\ell)=\min\{j:\|\XX_j\|\geq 2\ell \cdot k(n)\}$. Let $\Lambda_\ell$ be the
cube of side length $k(n)$ centered at $\XX_{j(\ell)}$.
By the weak law of large numbers, given the event
$|[\XX] \cap \Lambda_\ell| \ge \delta |\Lambda_\ell|$, the conditional probability
that $\Lambda_\ell$ satisfies the inequality (\ref{eq:incube}) is at least $1/2$.
Therefore, by Lemma
\ref{lem:largediv}, the
probability that none of the cubes $\Lambda_\ell$ with $\ell \in [1,\sqrt{n}]$ satisfy the
condition in (\ref{eq:incube}), is bounded by
\[
\left(1-\frac{\delta}{2} e^{-C(\log
n)^\frac{d-2}{d-1}}\right)^{\sqrt{n}}<e^{-n^\frac{1}{4}}
\]
The right-hand side is summable in $n$, so again by Borel-Cantelli we are done.
\end{proof}

\subsection{Oriented and biased random walk}
In this subsection we prove Part \ref{item:orient} of Theorem \ref{thm:zdcase}.
We start with a simple lemma that holds for general graphs and walks.

\begin{lemma}\label{lem:mart}
Let $G$ be a graph with a distinguished root
$v$, let $\Psi$ be a distribution on $\PP(G,v)$ and let $\Omega$ be the sample
space on which the measures $\mu$ and $\nu$ live.

Let $\kappa=(v_1,v_2,\ldots)$ be an ordering of the vertices in $G$. For $\omega\in\Omega^{V(G)}$ define
\begin{equation}\label{eq:omegn}
\omega^{(n)}=\{\eta\in\Omega^{V(G)}:\eta(v_i)=\omega(v_i),i=1,\ldots,n\}\subseteq\Omega^{V(G)}
\end{equation}
 and
\begin{equation}\label{eq:f}
f_n(\omega)=f_n^\kappa(\omega)= \frac {P\left(\omega^{(n)}\right)}
{Q\left(\omega^{(n)}\right)} \ \ \ \ \ \ \ ; \ \ \
\ \ \ \ f=f^\kappa=\lim_{n\to\infty} f_n^\kappa
\end{equation}
and
\begin{equation}\label{eq:g}
g_n(\omega)=g_n^\kappa(\omega)= \frac {Q\left(\omega^{(n)}\right)}
{P\left(\omega^{(n)}\right)} \ \ \ \ \ \ \ ; \ \ \
\ \ \ \ g=g^\kappa=\lim_{n\to\infty} g_n^\kappa
\end{equation}

Then,

\begin{enumerate}
\item The limit in (\ref{eq:f}) exists $Q$-almost surely and is the Radon-Nikodym derivative of (the absolute continuous part of) $P$ with respect to $Q$.
\item The limit in (\ref{eq:g}) exists $P$-almost surely and is the Radon-Nikodym derivative of (the absolute continuous part of) $Q$ with respect to $P$.
\item\label{item:ordf} For every two orderings $\kappa_1$ and $\kappa_2$, we have that $Q$-almost surely,
$f^{\kappa_1}=f^{\kappa_2}$.
\item\label{item:ordg}  For every two orderings $\kappa_1$ and $\kappa_2$, we have that $P$-almost surely,
$g^{\kappa_1}=g^{\kappa_2}$. 
\item\label{item:condabs}
\[
Q\ll P \ \Longleftrightarrow\ f>0 \ \ \ Q-\mbox{a.s.}
\]
and
\[
P\ll Q \ \Longleftrightarrow\ g>0 \ \ \ P-\mbox{a.s.}
\]
\item\label{item:condsing}
\[
Q\perp P \ \Longleftrightarrow\ f=0 \ \ \ Q-\mbox{a.s.} \
\Longleftrightarrow\ g=0 \ \ \ P-\mbox{a.s.}
\]
\end{enumerate}
\end{lemma}

Lemma \ref{lem:mart} follows from standard martingale techniques,
see e.g. Section 4.3.c of \cite{durrett}.

We now prove a simple lemma that is very useful in proving absolute
continuity. Variants of this lemma appeared in \cite{levin} and in
\cite{BS}. Similarly to the previous lemma, this lemma holds for
general graphs and (possibly non-Markovian) walks.
\begin{lemma}\label{lem:shnitzel}
Let $\XX^{(1)}$ and $\XX^{(2)}$ be two independent samples of the
measure $\Psi$ on $\PP(G,v)$. If there exists $C>0$ such that for
every $n$
\begin{equation}\label{eq:exptails}
(\Psi\times\Psi)\left(\left|\left[\XX^{(1)}\right]\cap\left[\XX^{(2)}\right]\right|>n\right)
\le e^{-Cn}
\end{equation}
then there  exist $\mu\neq\nu$ such that the measures $P$ and $Q$ are indistinguishable.
\end{lemma}

\begin{proof}
 Let $\mu\neq\nu$ be such that
\begin{equation}\label{eq:defgamma}
\zeta:=\int\left[\frac{d\nu(x)}{d\mu(x)}\right]^2d\mu(x)=\int\left[\frac{d\nu(x)}{d\mu(x)}\right]d\nu(x)<e^C \,,
\end{equation}

with $C$ as in (\ref{eq:exptails}).
 Let
$v_1,v_2,\ldots$ be the same ordering of the vertices in $G$ as in
Lemma \ref{lem:mart}. For a configuration $\omega\in\Omega^G$ we
look again at the  functions $g$ and $g_n$, defined in
(\ref{eq:g}). By Proposition \ref{prop:01law}, it suffices to prove that
$Q$ is absolutely continuous with respect to $P$; by Lemma \ref{lem:mart},
this is equivalent to uniform integrability
of the martingale $\{g_n\}$ (w.r.t. $P$). To establish this, we will show
that $\{g_n\}$ is bounded in $L^2(P)$.

\[
g_n(\omega)= \frac {Q\left(\omega^{(n)}\right)}
{P\left(\omega^{(n)}\right)}
=\int_{\PP(G,v)}
\frac {\qx\left(\omega^{(n)}\right)}
{P\left(\omega^{(n)}\right)}
d\Psi(\XX),
\]

where, as before,
\begin{equation}\label{eq:defqx}
\qx=\mu^{V(G)-[\XX]}\times\nu^{[\XX]}.
\end{equation}
Let
\[
g_n^{(\XX)}(\omega)=\frac
{\qx\left(\omega^{(n)}\right)}
{P\left(\omega^{(n)}\right)}.
\]

Then
\begin{eqnarray}
\nonumber \E_P(g_n^2)&=&
\int_{{\PP(G,v)}^2}\E_P\left[g_n^{(\XX^{(1)})}(\omega)\cdot
g_n^{(\XX^{(2)})}(\omega)\right]d\Psi(\XX^{(1)})d\Psi(\XX^{(2)})\\
 &=&
\int_{{\PP(G,v)}^2}
 \prod_{i=1}^n \E_P
\left[
\frac
{Q_{\XX^{(1)}}\left(\omega(v_i)\right)}
{P\left(\omega(v_i)\right)}
\cdot\frac
{Q_{\XX^{(2)}}\left(\omega(v_i)\right)}
{P\left(\omega(v_i)\right)}
\right]
d\Psi(\XX^{(1)})d\Psi(\XX^{(2)})
\label{eq:moshe}
\end{eqnarray}

For given $\XX^{(1)}$ and $\XX^{(2)}$, the product inside the integral in (\ref{eq:moshe}) naturally breaks into four products: for values of $i$ satisfying
\begin{eqnarray}
 i\in{\left[\XX^{(1)}\right]}^c\cap {\left[\XX^{(2)}\right]}^c \label{eq:kl1}, \\
 \nonumber\\
 i\in{\left[\XX^{(1)}\right]}^c\cap \left[\XX^{(2)}\right] \label{eq:kl2}, \\
 \nonumber\\
 i\in{\left[\XX^{(1)}\right]}\cap {\left[\XX^{(2)}\right]}^c \label{eq:kl3}, \\
 \nonumber\\
 \mbox{or }i\in\left[\XX^{(1)}\right]\cap \left[\XX^{(2)}\right]. \label{eq:kl4}
\end{eqnarray}

It is easy to see that for $i$ as in (\ref{eq:kl1}), (\ref{eq:kl2}) and (\ref{eq:kl3}),
\[
\E_P\left[ \frac {Q_{\XX^{(1)}}\left(\omega(v_i)\right)}
{P\left(\omega(v_i)\right)} \cdot\frac
{Q_{\XX^{(2)}}\left(\omega(v_i)\right)}
{P\left(\omega(v_i)\right)} \right]=1,
\]

while for $i$ as in (\ref{eq:kl4}),
\[
\E_P\left[ \frac {Q_{\XX^{(1)}}\left(\omega(v_i)\right)}
{P\left(\omega(v_i)\right)} \cdot\frac
{Q_{\XX^{(2)}}\left(\omega(v_i)\right)}
{P\left(\omega(v_i)\right)} \right]= \zeta
\]

so

\begin{eqnarray*}
\E_P\left[g_n^{(\XX^{(1)})}(\omega)\cdot g_n^{(\XX^{(2)})}(\omega)\right]
=\zeta^{\left|\left[\XX^{(1)}\right]\cap\left[\XX^{(2)}\right]\cap\{v_1,\ldots,v_n\}\right|}.
\end{eqnarray*}

Therefore, by the choice of $\zeta$ and by (\ref{eq:defgamma}),
\[
\sup_{n}\left\|g_n\right\|_2=\E_{\Psi\times\Psi}\left[\zeta^{\left|\left[\XX^{(1)}\right]\cap\left[\XX^{(2)}\right]\right|}\right]
<\infty\,.
\]

\end{proof}

The following is a corollary of the proof:
\begin{corollary}\label{cor:iff}
There exist distinct $\mu$ and $\nu$ such that $P$ and $Q$ are
indistinguishable and the Radon-Nikodym derivative is in $L^2$ if
and only if (\ref{eq:exptails}) holds for some $C$ and all $n$.
\end{corollary}

Next, we prove the $d\geq 4$ case of Part (\ref{item:orient}) of
Theorem \ref{thm:zdcase}. We start with the  special case of simple {\em oriented\/}
random walk where the increments give equal weight to the $d$ standard basis vectors.
 For this case, all we need is the following lemma
from Cox and Durrett~\cite{cox-durrett} (who attribute the idea to H. Kesten).

\begin{lemma}\label{lem:notmanyintersectorient}
Let $\XX^{(1)}$ and $\XX^{(2)}$ be two independent paths of a
nearest-neighbor random walk in $\Z^d$, $d\geq 4$ with the simple
oriented transition kernel. Then there exists $C>0$ such that
\[
\E\left[e^{C\left|\left[\XX^{(1)}\right]\cap\left[\XX^{(2)}\right]\right|}\right]<\infty.
\]
\end{lemma}
We recall the short proof for the reader's convenience.

\begin{proof}
for every $k$ we have $\|\XX^{(1)}_k\|_1=\|\XX^{(2)}_k\|_1=k$. Therefore,
if $\XX^{(1)}(j)=\XX^{(2)}(k)$ then $j=k$. Thus, $\left|\left[\XX^{(1)}\right]\cap\left[\XX^{(2)}\right]\right|$ is the
number of returns to zero of the Markov chain $\{\XX^{(1)}_k-\XX^{(2)}_k\}_{k=1}^\infty$. This Markov chain is a
$d-1$ dimensional random walk, and therefore is transient for $d\geq 4$. The lemma follows from the general fact that the number of
returns to the origin of a transient Markov chain is a geometric random variable.
\end{proof}

\vspace{0.5cm}

To prove Part (\ref{item:orient}) of Theorem \ref{thm:zdcase}
for all nearest neighbor walks with nonzero mean, the following more general lemma is needed.
\begin{lemma}\label{lem:notmanyintersect}
Let $\XX^{(1)}$ and $\XX^{(2)}$ be two independent paths of a
nearest-neighbor random walk in $\Z^d$, $d\geq 4$ with non-zero
mean. Then there exists $C>0$ such that
\[
\E\left[e^{C\left|\left[\XX^{(1)}\right]\cap\left[\XX^{(2)}\right]\right|}\right]<\infty.
\]
\end{lemma}

Lemma \ref{lem:notmanyintersect} is a special case of the first part of Theorem 2.4 of \cite{BS},
 and all proofs of the lemma that we know are difficult.

\medskip

\noindent
Next, we establish the case $d\leq 3$ of Part (\ref{item:orient}) of Theorem \ref{thm:zdcase}.
 We will use a simple counting argument.
Let $m$ be the drift of the random walk, and assume w.l.o.g. that
$\langle m,e_1\rangle>0$ and that $\langle m,e_1\rangle\geq|\langle m,e_i\rangle|$ for every $i$.
 Given $n$, let
\[
\hfC(n)=\{x\in(n/2,n]\times[-n,n]^{d-1}:\exists_{k\geq 0}
\|x-km\|_1<n^{1/2}\}.
\]

We will use the following statement in order to establish singularity of $P$ and $Q$:
\begin{claim}\label{claim:oriencount}
Let $\XX$ be a nearest-neighbor random walk in $\Z^d$ with mean $m \ne 0$.
if $d\leq 3$, then there exists $\rho>0$ such that for every $n$ large
enough,
\[
\Psi\left(\frac{\left|[\XX]\cap
\hfC(n)\right|}{\sqrt{\left|\hfC(n)\right|}}>\rho\right)>\rho.
\]
\end{claim}

\begin{proof}
\[
U(n) := \left|[\XX]\cap (n/2,n]\times[-n,n]^{d-1}\right| \geq
\left|[\XX]\cap \hfC(n)\right|
\]
satisfies
\begin{equation}\label{eq:joseph}
\Psi(U(n)>an)<e^{-\theta n}
\end{equation}
for $a>\langle m,e_1\rangle^{-1}$ and $\theta=\theta(a)>0$. On the
other hand,
\begin{eqnarray}\label{eq:cltC}\nonumber
\E_\Psi(\left|[\XX]\cap \hfC(n)\right|) &=&
 \sum_{i=1}^\infty\Psi\left[\XX(i)\in\hfC(n)
  \ \& \ \XX(j)\neq\XX(i)\mbox{ for all }j>i\right] \\
&=&\gamma\sum_{i=1}^\infty\Psi\left[\XX(i)\in\hfC(n)\right] \, \ge c_1 n \,.
\end{eqnarray}
where $\gamma$ is the escape probability of the random walk. To
see that the last inequality in (\ref{eq:cltC}) holds, note that for
$\frac{5}{8\langle m,e_1\rangle} n<i<\frac{7}{8\langle m,e_1\rangle}
n$,
\begin{eqnarray*}
\Psi\left(\XX(i)\in\hfC(n)\right) \ge \Psi\left(\|\XX(i)-E(\XX(i))\|_1<\sqrt{n} \right)
\ge c_0>0.
\end{eqnarray*}

Note that $|\hfC(n)|=O(n^{1+\frac{d-1}{2}})$ and thus $|\hfC(n)|=O(n^2)$
for $d\leq 3$. In conjunction with (\ref{eq:cltC}) and (\ref{eq:joseph}),
we deduce the existence of positive $\rho$ such that
\[
\Psi\left(\frac{\left|[\XX]\cap
\hfC(n)\right|}{\sqrt{\left|\hfC(n)\right|}}>\rho\right)>\rho,
\]
as desired.
\end{proof}

\begin{proof}[Proof of singularity for $d\leq 3$]
Let $n_k=2^k$. Let $A_k$ be the event
\[
A_k=\left\{\frac{\left|\XX\cap
\hfC(n_k)\right|}{\sqrt{\left|\hfC(n_k)\right|}}>\rho\right\}.
\]

Let $\xi\in\Omega$ be s.t. $\nu(\xi)>\mu(\xi)$.

Then for every $k$, Let $B_k$ be the event
\[
B_k=
 \left\{\#\left\{x\in \hfC(n_k):\omega(x)=\xi
 \right\}\geq\mu(\xi)\left[|\hfC(n_k)|-\rho\sqrt{|\hfC(n_k)|}\right]+\rho\nu(\xi)\sqrt{|\hfC(n_k)|}
 \right\}
\]

Let  $\tQ$ be the law of the pair $(\XX,\omega)$, where $\XX\in
\PP(G,v)$ is a random path sampled from $\Psi$ and $\omega$ is a
random scenery sampled from $Q_\XX$. In other words, $\tQ$ is a
Borel measure on $\PP(G,v)\times\Omega^{V(G)}$, and for Borel sets
$\Phi \subset \PP(G,v)$ and $A \subset \Omega^{V(G)}$, it satisfies
\begin{equation}\label{eq:tp}
\tQ(\Phi \times A)=\int_\Phi Q_\XX(A) \,d\Psi(\XX).
\end{equation}

Then (under both $P\times\Psi$ and $\tQ$) the $B_{k}$-s are
independent conditioned on $\XX$, and for all $k$ large, by the
central limit theorem and by stochastic domination,

\begin{equation}\label{eq:QPBk}
\tQ(B_{k}|A_k)\geq1/2 \ \ \ ; \ \
\gamma:=\lim_{k\to\infty}P(B_{k})<1/2 \ \ \ ; \ \ \tQ(B_k|\XX)\geq
P(B_{k})\ \  \Psi-\mbox{a.s.}
\end{equation}

$\Psi(A_k)\geq \rho$ for all $k$ large enough, and therefore there
exists $\tau>0$ such that
\begin{equation}\label{eq:forth}
 \Psi\left(
 \limsup_{k\to\infty} \frac{1}{k}\sum_{j=1}^k{\bf 1}_{A_j}
 \geq\frac{\rho}{2}
 \right)\geq\tau
\end{equation}
(This follows from, e.g, Lemma 4.2 of \cite{noam} referring to the
events  $\frac{1}{k}\sum_{j=1}^k{\bf 1}_{A_j}\geq\frac{\rho}{2}$.)

Let $Z$ be the event in (\ref{eq:forth}), and let
\[
 W=\left\{
 \limsup_{k\to\infty} \frac{1}{k}\sum_{j=1}^k{\bf 1}_{B_j}
 \geq \tau\cdot\frac{1}{2}+(1-\tau)\cdot\gamma
 \right\}.
\]

Then by (\ref{eq:QPBk}) and independence, $P(W)=0$. On the other
hand, $\tQ(W|Z)>0$ and so  $Q(W)\geq \tQ(W|Z)\Psi(Z)>0$. Therefore
$Q$ and $P$ are not mutually absolutely continuous, and by
Proposition \ref{prop:01law} they are singular.
\end{proof}

\subsection{Non-Markovian paths}
Here we supply the proofs of parts \ref{item:genpath3} and
\ref{item:genpath2} of Theorem \ref{thm:zdcase}.

\begin{proof}[Proof of part \ref{item:genpath3} of Theorem \ref{thm:zdcase}]
Based on Lemma \ref{lem:shnitzel} all we need to show is the
existence of a measure on paths satisfying the exponential
intersection tail property in $\Z^3$. Such measures were constructed in
\cite{BPP} and \cite{HM}.
\end{proof}

\begin{proof}[Proof of part \ref{item:genpath2} of Theorem \ref{thm:zdcase}]
Let $f$ be a function from $\Omega$ to $\R$ such that $\E_\mu(f)=0$
and $\E_\nu(f)=1$. Let $\{J_k\}_{k=1}^\infty$ be a sequence of
weights, and let $L_k=\{x \in \Z^2 :\|x\|_1=k\}$. Define
\[
U_n=  \sum_{k=1}^n J_k \sum_{v \in L_k} f(\omega(v)) \,.
\]
Then $\E_P(U_n)=0$. Moreover, the measure $\tQ$ defined in (\ref{eq:tp}) satisfies
\[\E_\tQ(U_n|\XX)=\sum_{k=1}^n J_k \Bigl|[\XX]\cap L_k \Bigr| \geq \sum_{k=1}^n J_k  \] \,
and (since $|L_k|=4k$),
$$
 \var_\tQ(U_n|\XX)= \sum_{k=1}^n
 J_k^2\Bigl[\, \, \Bigl|[\XX]\cap L_k\Bigr|\var_\nu(f)+ \Big|L_k \setminus[\XX]\Big|\var_\mu(f)\Bigr]
  \leq 4C\sum_{k=1}^n k J_k^2
$$
where $C=\max(\var_\mu(f),\var_\nu(f))$. We also have
\[
\var_P(U_n)\le 4C\left(\sum_{k=1}^n k J_k^2 \right) \,.
\]

Pick $J_k=k^{-1}$, so that
$$ 
\frac
 {\left[\sum_{k=1}^n J_k\right]^2}
 {\sum_{k=1}^n k J_k^2}
 \longrightarrow\infty
$$ 

Then by Chebyshev's inequality,
\[
P\left(U_n>\frac 12\sum_{k=1}^nJ_k\right)\to 0
\]
and
\[
Q\left(U_n>\frac 12\sum_{k=1}^nJ_k\right)\to 1 \, ,
\]
so the proof is complete.
\end{proof}


 \vspace{1cm}

\section{General graphs}\label{sec:gengrp}
In this section we prove Proposition \ref{prop:01law} and Theorem
\ref{thm:gengrp}. We start with Proposition \ref{prop:01law}. We
note that for the purpose of the proof given here, the assumption of
transience in the statement of the proposition can be relaxed to
assuming infinite orbits. However the question is only of interest
in the transient case.

Let $u$ be a neighbor of $v$, and  define $Q^*$ the way $Q$ is
defined, but with the path starting at $u$ instead of $v$. We
define the functions $f^*$ and $g^*$ similarly to $f$ and $g$
(recall (\ref{eq:f}) and (\ref{eq:g})), using $Q^*$ instead of
$Q$. The next lemma follows from the fact that $\mu$ and $\nu$ are
absolutely continuous with respect to each other.
\begin{lemma}\label{lem:qstar}
The measures $Q$ and $Q^*$ are absolutely continuous with respect
to each other. In particular, $f(\omega)=0$ if and only if
$f^*(\omega)=0$ for $Q$-almost every $\omega$, and $g(\omega)=0$
if and only if $g^*(\omega)=0$ for $P$-almost every $\omega$.
\end{lemma}

Using Lemma \ref{lem:mart} we are now ready to prove Proposition
\ref{prop:01law}.

\begin{proof}[Proof of Proposition \ref{prop:01law}]
We consider the following coupling of $P$ and $Q$: our sample
space is
\[
\Xi=\left(\Omega^{V(G)}\right)^2\times\PP(G,v)
\]
The measure on this space is defined as follows: the first copy of
$\Omega^{V(G)}$ is equipped with the measure $\mu^{V(G)}$, the
second copy with the measure $\nu^{V(G)}$ and $\PP(G,v)$ is equipped
with the measure $\Psi$ on paths determined by $M$ and the starting
point $v$. These three are chosen to be independent of each other.
An element of $\Xi$ is denoted  $\eta=(\eta_1,\eta_2,\XX)$. We
now define $\omega_1(\eta)$ and $\omega_2(\eta)$ as follows: for a
vertex $u\in {V(G)}$,
\[
\omega_1(u)=\eta_1(g),
\]
and
\[
\omega_2(u)=\left\{
\begin{array}{ll}
\eta_1(u) & \mbox{ if } u\notin[\XX]\\
\eta_2(u) & \mbox{ if } u\in[\XX]
\end{array}
\right..
\]

We define $\tilde{f}(\eta):=f(\omega_2(\eta))$ and
$\tilde{g}(\eta):=g(\omega_1(\eta))$. In light of Parts
\ref{item:condabs} and \ref{item:condsing} of  Lemma
\ref{lem:mart}, all we need in order to prove the proposition is
to find a measure preserving ergodic transformation
$T:\Xi\to\Xi$ such that the events
$A_1=\{\tilde{f}(\eta)>0\}$ and $A_2=\{\tilde{g}(\eta)>0\}$ are
$T$-invariant.

We proceed with the definition of the transformation $T$. For
every $u\in G$, let $\alpha_u$ be an $M$-preserving automorphism
of $G$ such that $\alpha_u(u)=v$. the map $\alpha_u$ exists by the
assumption that $M$ is invariant under a transitive subgroup of
the automorphism group of $G$.

The path $\XX$ is a function from $\N$ to ${V(G)}$ with $\XX(0)=v$.
Let $\alpha=\alpha_{\XX(1)}$. Then,
\begin{enumerate}
\item for $n\in\N$,
\[
\left(T(\XX)\right)(n):=\alpha(\XX(n+1)).
\]
\item For $u\in {V(G)}$ and $i\in\{1,2\}$,
\[
\left(T(\eta_i)\right)(u):=\eta_i\left(\alpha^{-1}(u)\right).
\]
\end{enumerate}

It is easy to see that $T$ is measure preserving. The fact that
$A_1$ and $A_2$ are $T$-invariant follows from Lemma
\ref{lem:qstar} and parts \ref{item:ordf} and \ref{item:ordg} of
Lemma \ref{lem:mart}. We now show that $T$ is mixing, and
therefore ergodic. Let $A$ and $B$ be cylinder sets that depend
only on the first $r$ steps of $\XX$ and on $\eta_1$ and $\eta_2$
in the ball $B(v,r)$. Then,
\[
\Phi(T^{-n}A\cap B)-\Phi(A)\Phi(B)\leq\prob(\XX(n)\in B(v,2r))
\mathop{\longrightarrow}^{n\to\infty}0.
\]
For general sets, we get this by approximating them with cylinder
sets.
\end{proof}

Now we turn to proving Theorem \ref{thm:gengrp}. We first need a
lemma which is reminiscent of Lemma \ref{lem:shnitzel}. This lemma
is in the same spirit as Lemma 7.1 in \cite{levin}.
\begin{lemma}\label{lem:quenchedshnitzel}
Let $\XX^{(1)}$ and $\XX^{(2)}$ be two independent samples of the
random walk path. If there exists $C$ such that
\begin{equation}\label{eq:qsh}
\prob\left(\E\left[\left.e^{C\left|\left[\XX^{(1)}\right]\cap
\left[\XX^{(2)}\right]\right|}\right|\XX^{(1)}\right]<\infty\right)>0
\end{equation}

then there exist $\mu\neq\nu$ such that $P$ and $Q$ are
indistinguishable.
\end{lemma}

\begin{proof}

Using Proposition \ref{prop:01law}, all we need to show is that
(with the notations of (\ref{eq:f}) and (\ref{eq:g}))
\begin{equation*}
Q\left(\lim_{n\to\infty}f_n>0\right)>0.
\end{equation*}
This is equivalent to saying
\begin{equation*}
Q\left(\lim_{n\to\infty}g_n<\infty\right)>0.
\end{equation*}


Let $\tQ$
be as in (\ref{eq:tp}). What we need to show is the same as
\begin{equation*}
\tQ\left(\lim_{n\to\infty}g_n<\infty\right)>0.
\end{equation*}

It is sufficient to show that there exists an event $A$ of
positive probability which is determined by $\XX$ satisfying

\begin{equation*}
\lim_{n\to\infty}\E_{\tQ}(g_n\cdot{\bf 1}_A)<\infty
\end{equation*}

\renewcommand{\tp}{{(P\times\Psi)}}
\renewcommand{\tP}{{(P\times\Psi)}}

which, using the fact that $g_n$ is the derivative $\frac{dQ}{dP}$
conditioned on $\omega^{(n)}$ and that $g_n$ and $\XX$ are
independent conditioned on $\omega$, translates to


\begin{equation}\label{eq:posfin}
\lim_{n\to\infty}\E_{\tP}(g_n^2\cdot{\bf 1}_A)<\infty.
\end{equation}


We now repeat the calculations from the proof of Lemma
\ref{lem:shnitzel}:

\begin{eqnarray}
\nonumber
 \E_\tP(g_n^2 | \XX^{(1)})=
 \int_{\PP(G,v)}
 \E_P \left[\left.
 \prod_{i=1}^n
\frac
{Q_{\XX^{(1)}}\left(\omega(v_i)\right)}
{P\left(\omega(v_i)\right)}
\cdot\frac
{Q_{\XX^{(2)}}\left(\omega(v_i)\right)}
{P\left(\omega(v_i)\right)}
 \ \right|\XX^{(1)},\XX^{(2)}
 \right]
 d\Psi(\XX^{(2)})
\label{eq:moshe2}
\end{eqnarray}

We use the same decomposition as in the proof of Lemma \ref{lem:shnitzel}  to get that
\[
\E_\tP(g_n^2 | \XX^{(1)})= \E_{\Psi}\left(\left.
\zeta^{\left|[\XX^{(1)}]\cap[\XX^{(2)}]\cap\{v_1,\ldots,v_n\}\right|} \right|\XX^{(1)} \right)
\]

Let
\[
 U=U(\XX^{(1)})=\E_\Psi\left[\left.\zeta^{\left|\left[\XX^{(1)}\right]\cap\left[\XX^{(2)}\right]\right|}\right|\XX^{(1)}\right].
\]

Let $M$ be a large finite number such that $\Psi(U<M)>0$, and let
$A=(U<M)$.

For every choice of $M$, the sequence $\E(g_n^2\cdot{\bf
1}_{U<M})$ is a bounded sequence, and therefore

\begin{eqnarray*}
P\left(\lim_{n\to\infty}g_n^2\cdot{\bf
1}_{A}<\infty\right)=P(U<M)>0,
\end{eqnarray*}

and (\ref{eq:posfin}) holds.
\end{proof}

\begin{proof}[Proof of part \ref{item:pos_speed} of Theorem \ref{thm:gengrp}]
Part \ref{item:pos_speed} of Theorem \ref{thm:gengrp} will follow
from Lemma \ref{lem:quenchedshnitzel} once we prove the following
claim:
\begin{claim}\label{claim:itai}
Let $G$ be a Cayley graph such that the speed of the simple random
walk on $G$ is positive, and let $\XX^{(1)}$ and $\XX^{(2)}$ be two
independent samples of the path of the random walk on $G$ started
at the same point $v$. Then there exist $C$ such that
\[
\prob\left( \E\left[\left. e^{C\left|\left[\XX^{(1)}\right]\cap\left[\XX^{(2)}\right]\right|}  \ \right|\ \XX^{(1)}
\right] <\infty \right)=1.
\]
\end{claim}

\end{proof}

\begin{proof}[Proof of Claim \ref{claim:itai}]
By Proposition 6.2 of \cite{itai},
when the speed is positive, almost surely there exists $\gamma>0$ such that
$
G(\XX^{(1)}(n))<e^{-n\gamma}
$
for all $n$ large enough, where $G$ is Green's function for the
random walk started at $v$. Since
$\prob(x\in\left[\XX^{(2)}\right])\leq G(x)$, we infer that almost surely,

\begin{eqnarray*}
\prob\left(\left.\left|\left[\XX^{(1)}\right]\cap\left[\XX^{(2)}\right]\right|>n\ \right|\XX^{(1)}\right)
 \leq\sum_{\ell>n}\prob\left(\XX^{(1)}(\ell)\in\left[\XX^{(2)}\right]\right)
 \leq\sum_{\ell>n}G(\XX^{(1)}(\ell))
\end{eqnarray*}
and the right-hand side is at most $O \left(e^{-n\gamma}\right)$.

\end{proof}

\begin{proof}[Proof of part \ref{item:non_am} of Theorem \ref{thm:gengrp}]
Here we use Lemma \ref{lem:shnitzel} and the following claim:
\begin{claim}
Let $G$ be a transitive nonamenable graph, and let $\XX^{(1)}$ and
$\XX^{(2)}$ be two independent samples of the path of the random walk
on $G$. Then there exist $K$ such that
\[
\E\left( e^{K\left|\left[\XX^{(1)}\right]\cap\left[\XX^{(2)}\right]\right|}\right) <\infty
\]
\end{claim}

\ignore{
\begin{proof}
Let $B_n=\{u:2^n<d(u,v)\leq 2^{n+1}\}$. First we see that there
exist $C<\infty$ and $\rho>0$ such that
\begin{equation}\label{eq:inanu}
\E\left(|\XX_1\cap\XX_2\cap B_n|\right) < C e^{-\rho 2^n}.
\end{equation}

The reason is that there exist $\gamma>0$ such that for every
$u\in B_n$,
\[
\prob(u\in \XX_1)<e^{-\gamma n}
\]
and
\[
\sum_{u\in B_n}\prob(u\in \XX_1)=O(n).
\]
Therefore
\[
\E\left(|\XX_1\cap\XX_2\cap B_n|\right) = \frac{O(n^2)}{e^{\gamma
n}} < Ce^{-\rho 2^n}.
\]
when $\rho<\gamma$ and $C$ is sufficiently large.

From (\ref{eq:inanu}) we get that
\begin{equation}\label{eq:outball}
\prob\left(\XX_1\cap\XX_2\cap \{u:d(u,v)>2^n\}\neq\emptyset\right)
< 2Ce^{-\rho 2^n}.
\end{equation}

In addition, there exist $D<\infty$ and $\lambda>0$ such that
\[
\prob(\XX_1\cap\{u:d(u,v)<2^n\}>D2^n)<e^{-\lambda 2^n}
\]

Now for every $k$, let $n$ be so that $D2^n<k\leq D2^{n+1}$, and
then
\begin{eqnarray*}
\prob\left(\left|\XX_1\cap\XX_2\right|>k\right) &\leq&
\prob\left(\left|\XX_1\cap\XX_2\right|>D2^n\right)\\
&\leq& \prob(\XX_1\cap\{u:d(u,v)<2^n\}>D2^n) \\ &+&
\prob\left(\XX_1\cap\XX_2\cap \{u:d(u,v)>2^n\}\neq\emptyset\right)
\\ &=&e^{-\Omega(2^n)}=e^{-\Omega(k)}
\end{eqnarray*}
as desired.
\end{proof}
}
\begin{proof}
For any $n$ and $\epsilon>0$,
\begin{equation}\label{eq:nonam}
 \prob\left[\XX^{(1)}(n)\in\left[\XX^{(2)}\right]\right]
 \leq\prob\left[\XX^{(1)}(n)\in B(v,\epsilon n)\right]
 +\max_{z}\sum_{t>\epsilon n}\prob\left[\XX^{(2)}(t)=z\right]
\end{equation}
By non-amenability, for small enough $\epsilon$ both summands in
the RHS of (\ref{eq:nonam}) decay exponentially with $n$, so
\begin{equation*}
 \prob\left[\XX^{(1)}(n)\in\left[\XX^{(2)}\right]\right]
 \leq Ce^{-n\gamma}
\end{equation*}
for some (non-random) $C$ and $\gamma$. From here,
\[
 \prob\left[\left|\left[\XX^{(1)}\right]\cap\left[\XX^{(2)}\right]\right|>n\right]
 \leq\sum_{\ell>n}\prob\left[\XX^{(1)}(\ell)\in\left[\XX^{(2)}\right]\right]
 \leq \frac{C}{1-e^{-\gamma}}e^{-n\gamma}.
\]
\end{proof}
\end{proof}

\section{Finding the threshold on trees}\label{sec:tree}

In this section we prove Theorems \ref{thm:branum} and \ref{thm:sinpath}.
Recall the definitions of relative entropy, branching number and
local dimension (Definitions \ref{def:relent}, \ref{def:bn} and
\ref{def:locdim} in the Introduction).


We define a {\bf cut} in a tree to be a subset $\cut\subseteq V(T)$
such that the connected component of the root in $V(T)-\cut$ is
finite. We only consider minimal cuts, i.e. cuts such that removal
of one point will connect the root to infinity. From the
Min-cut-max-flow theorem, one can deduce the following
characterization of the branching number, see \cite{russell}
for the proof.
\begin{lemma}\label{lem:cutbran}
Let $\C(T)$ be the set of cuts of $T$. Then $b(T)$ is the infimum of
all values $\beta$ such that
\[
\inf_{\cut\in\C(T)}\sum_{u\in \cut}\beta^{-|u|}=0.
\]
\end{lemma}


The proofs of Theorems \ref{thm:branum} and \ref{thm:sinpath} are
fairly similar, and therefore we start with two lemmas that are at
the core of the proofs of both theorems.

In what follows, for a probability measure $\Psi$ on $\partial T$, we take $\Psi(\{v_1,v_2,\ldots,v_n\})$
to be the measure
of the set of rays going through any of the $v_i$-s. Additionally,
the measure of a finite self-avoiding path starting at the root is the measure of the set of all of
its extensions to infinite self-avoiding paths. A set of vertices $V_0 \subset V(T)$ is called an
{\bf antichain} if for any pair of vertices $v,w \in V(T)$ such that $v$ is an ancestor of $w$, at most one of $v,w$ is in $V_0$.

\begin{lemma}\label{lem:sub}
Let $H=H(\nu|\mu)$ and let $\Psi$ be a measure on $\partial T$.
Assume that there exist disjoint antichains $V_n\subseteq V(T)$ such that
\[
\lim_{n\to\infty}\Psi(V_n)=1
\]
and
\begin{equation}\label{eq:limens}
\lim_{n\to\infty} \sum_{u\in V_n} e^{-|u|H} = 0.
\end{equation}

Then $P$ and $Q$ are singular.
\end{lemma}

\begin{lemma}\label{lem:super}
Let $H=H(\nu|\mu)$ and let $\Psi$ be a measure on $\partial T$.
If there exists $\gamma>0$ such that for $\Psi$ almost every $\XX$,
\begin{equation}\label{eq:limefes}
\lim_{ \tiny
\begin{array}{c} u\in[\XX] \\ |u|\to\infty\end{array}}
e^{(H+\gamma)|u|}\Psi(u)=0
\end{equation}
then $P$ and $Q$ are indistinguishable.
\end{lemma}

\begin{proof}[Proof of Lemma \ref{lem:sub}]
Let
 \[
 \vo(V_n):=\sum_{u\in V_n} e^{-|u|H}.
 \]
 Let $n_k$ be a (deterministic) subsequence satisfying
\begin{equation}\label{eq:defnk}
\sum_{k=1}^\infty \vo(V_{n_k}) < \infty.
\end{equation}
For $\rho\in\Omega$, let $r(\rho)=\frac{\nu(\rho)}{\mu(\rho)}$. Let
$\tQ$ be the measure defined on $\Omega^{V(T)}\times\partial T$ as in (\ref{eq:tp}). Remember that $Q$ is the
$\Omega^{V(T)}$ marginal of $\tQ$.
Let $(\omega,\bar{\XX})$ be a sample of $\tQ$.
Then $\bar{\XX}$ is distributed according to $\Psi$. We denote the elements of $\bar{\XX}$  by $u_1,u_2,\ldots$. Then the sequence $\{\omega{(u_n)}\}_{n=1}^\infty$ is i.i.d. $\nu$ and independent of $\bar{\XX}$.

For every $u\in V(T)$, let $\XX_u$ be the (finite) path from the
root to $u$, and we use $K_u$ to denote the event
\[
 K_u=\left\{\prod_{z\in
\XX_{u}} r(\omega(z)) \geq e^{nH} \right\}.
\]

By the central limit theorem,


\begin{eqnarray}\label{eq:mar}
\nonumber \lim_{n\to\infty}\tQ(K_{{u_n}})
&=&\lim_{n\to\infty}\tQ\left( \prod_{z\in
\XX_{u_n}} r(\omega(z)) \geq e^{nH} \right)\\
&=&\lim_{n\to\infty}\tQ\left(  \sum_{z\in \XX_{u_n}} \log
r(\omega(z)) \geq {nH} \right) = 1/2.
\end{eqnarray}

In addition, by the condition that $\Psi(V_n)\to 1$, we get that
$\tQ([\bar{\XX}]\cap V_n\neq\emptyset)\to 1$. From here we get that
almost surely the set
$
B:=\{k:[\bar{\XX}]\cap V_{n_k}\neq\emptyset\}
$
satisfies $|B|=\infty$. Let $z_1,z_2,\ldots$ be the points on $\bar{\XX}$
that are also in $\cup V_{n_k}$. From (\ref{eq:mar}), we get


\begin{equation}\label{eq:vnhetzi}
\liminf_{k\to\infty}\tQ\left( K_{z_k} \right)\geq 1/2.
\end{equation}

\ignore{ For fixed $n_1$ and $n_2\to\infty$,

\[
\tQ\left( \left. \prod_{u\in
\bar{\XX}_{u_{n_2}}} r(\omega(u)) \geq e^{nH}\
 \right |
\omega(u_1),\ldots,\omega(u_{n_1})
\right)\to 1/2.
\]
} The event that there exist infinitely many values of $k$ such that
$K_{z_k}$ holds
is a tail event on the values of $w$ along $\bar{\XX}$, and is independent
of $\bar{\XX}$, and therefore is a zero-one event. By(\ref{eq:vnhetzi}) it
has positive $\tQ$-probability and therefore has $\tQ$-probability
one.

So $\tQ$-almost surely (and also $Q$-almost surely),
there exist infinitely many values of $k$ such that
\begin{equation}\label{eq:manyn1}
\exists_{u\in V_{n_k}} \mbox{ s.t. }
K_u \mbox{ holds}
\end{equation}

Let $u$ be a vertex of distance $n$ from the root.
\[
\E_P\left( \prod_{z\in \XX_u} r(\omega(z)) \right)=1,
\]

and therefore by Markov's inequality,
\[
P(K_u)=P\left( \prod_{z\in \XX_u} r(\omega(z))\geq e^{nH}\right)\leq
e^{-nH}=e^{-|u|H}
\]

and therefore
\[
P\left( \exists_{u\in V_n} K_u \right)\leq
\vo(V_n).
\]

So by Borel-Cantelli, $P$-almost surely, only finitely many values
of $k$ satisfy (\ref{eq:manyn1}). Therefore $P$ and $Q$ are
singular.
\end{proof}

\begin{proof}[Proof of Lemma \ref{lem:super}]
Recall that in the proofs of Lemmas \ref{lem:shnitzel} and
\ref{lem:quenchedshnitzel}, we had
\[
f_n(\omega)=\frac{P\left(\omegan\right)}
{Q\left(\omegan\right)}
\]
and
\[
g_n(\omega)=\frac{Q\left(\omegan\right)}
{P\left(\omegan\right)}
=\frac{1}{f_n(\omega)}
\]

where $\omegan$ is as in (\ref{eq:omegn}).

As before, it is sufficient to show that
\begin{equation}\label{eq:limfi1}
Q(\lim_{n\to\infty}g_n<\infty)=1
\end{equation}
and
\begin{equation}\label{eq:limfi2}
P(\lim_{n\to\infty}f_n<\infty)=1.
\end{equation}

From the fact that $1/g_n$ is a positive $Q$-martingale and
$1/f_n$ is a positive $P$-martingale we learn that the limits exist almost surely, but we must still
show that they are finite.

First we show (\ref{eq:limfi1}).

Fix $\epsilon>0$. Let $\delta$ be so that
\begin{equation}\label{eq:choicedelta}
\delta\sum_{\rho\in\Omega}\left|\log\left(\frac{\nu(\rho)}{\mu(\rho)}\right)\right|<
 \frac{\gamma}2.
\end{equation}
and let $N=N_\delta$ be so that for an i.i.d. $\nu$ sequence
$\{\ell_i\}$,
\[
\prob \left( \begin{array}{c}\mbox{for every }n>N,\mbox{ for every
}{\rho\in\Omega, } \\ \left|\frac{\#\{1\leq i\leq n:\ell_i=\rho
\}}{n}-\nu(\rho)\right|<\delta \end{array} \right)>1-\epsilon,
\]

and, using (\ref{eq:limefes}),

\begin{equation}\label{eq:axx}
\Psi\left\{ \XX : \mbox{for every }n>N, \ \
\Psi(\XX_n)<e^{-n(H+\gamma)} \right\}>1-\epsilon
\end{equation}

where $\XX_n$ is the $n$-th vertex of the path $\XX$.

%


For $\XX\in\partial T$, we define $A_\XX\subseteq\Omega^{V(T)}$ as follows:

If there exists $n>N$ such that $\Psi(\XX_n)\geq e^{-n(H+\gamma)}$ then $A_\XX=\emptyset$. Otherwise,
we take

\[
A_\XX=\left\{
\omega\,:\,
 \forall_{n>N}  \forall_{\rho\in\Omega} \left|\frac{\#\{1\leq i\leq
n:\omega(\XX_i)=\rho \}}{n}-\nu(\rho)\right|<\delta \right\}.
\]

We define $A\subseteq\partial T\times\Omega^{V(T)}$ to be
\[
A=\bigcup_{\XX\in\partial T}A_\XX\times\{\XX\}.
\]

$\tq(A)>1-2\epsilon$ by the choice of $N$ ($A$ is clearly measurable) .

We will show that
\begin{equation}\label{eq:condtopr}
\lim_{n\to\infty}\E_{\tq}(g_n(\omega)\cdot{\bf 1}_A)<\infty.
\end{equation}

Observe that (\ref{eq:limfi1}) follows from (\ref{eq:condtopr}).
To verify  (\ref{eq:condtopr}), compute

\begin{eqnarray*}
\E_{\tq}(g_n\cdot{\bf 1}_A) =\int g_n\cdot{\bf 1}_A d\tq =
\int\E_{\qx}(g_n\cdot{\bf 1}_{A_\XX})d\Psi(\XX),
\end{eqnarray*}

 where the second inequality follows from \eqref{eq:tp}, and
\ignore{
\begin{eqnarray}\nonumber
\E_{\qx}(g_n\cdot{\bf 1}_{A_\XX})
\E_{\tq}(g_n\cdot{\bf 1}_A|\XX)
 &=&\E_{\qx}\left(\frac{Q\left(\omegan\right)}
{P\left(\omegan\right)}
 \cdot{\bf 1}_A\ \right)\\ \nonumber
 &=&\E_{\qx}\left(\int
 \frac
 {{Q_{\XX^\prime}}\left(\omegan\right) }
 {P\left(\omegan\right)}
 d\Psi(\XX^\prime)
 \cdot{\bf 1}_A\ \right)\\ \label{eq:kmoshn}
 &=&\E_{\qx}\left(\int \prod_{u\in[\XX]\cap[\XX^\prime]}\frac{\nu(\omega(u))}{\mu(\omega(u))} d\Psi(\XX^\prime)
 \cdot{\bf 1}_A\ \right)\\
 \nonumber
 &=& \int \prod_{u\in[\XX]\cap[\XX^\prime]}r(\omega(u))\cdot {\bf 1}_A d\qx(\omega)d\Psi(\XX^\prime)
 \end{eqnarray}
}

\begin{eqnarray}\nonumber
\E_{\qx}(g_n\cdot{\bf 1}_{A_\XX})
 &=&\int\frac{Q\left(\omegan\right)}
{P\left(\omegan\right)}
 \cdot{\bf 1}_{A_\XX}d\qx(\omega)\\ \nonumber
 &=&\int\int
 \frac
 {{Q_{\XX^\prime}}\left(\omegan\right) }
 {P\left(\omegan\right)}
 d\Psi(\XX^\prime)
 \cdot{\bf 1}_{A_\XX}\ d\qx(\omega)\\ \nonumber
 &=&\int\int \prod_{u\in[\XX]\cap[\XX^\prime]\cap \{v_1,\ldots,v_n\}}\frac{\nu(\omega(u))}{\mu(\omega(u))} d\Psi(\XX^\prime)
 \cdot{\bf 1}_{A_\XX}\ d\qx(\omega)\\
 \label{eq:kmoshn}
 &=& \int \prod_{u\in[\XX]\cap[\XX^\prime]\cap \{v_1,\ldots,v_n\}}r(\omega(u))\cdot {\bf 1}_{A_\XX} d\qx(\omega)d\Psi(\XX^\prime)
 \end{eqnarray}

where the second equality follows from the decomposition $Q(W)
=\int{Q_{\XX^\prime}}(W)d\Psi(\XX^\prime)$ for every event $W\subseteq\Omega^{V(T)}$. The third equality then follows from the same reasoning as in the proof of Lemma \ref{lem:shnitzel}. Let $\hat{R}=\max\{r(\rho)\,:\,\rho\in\Omega\}$. By (\ref{eq:choicedelta}) and the choice of the event $A$, on the event $A$ we get that for every $\XX^\prime$, if $|[\XX]\cap[\XX^\prime]|
>N$
then
\begin{eqnarray*}
\prod_{v\in[\XX]\cap[\XX^\prime]\cap \{v_1,\ldots,v_n\}}r(\omega(v))
&\leq&\max\{\hat{R}^N,e^{\bigl|[\XX]\cap[\XX^\prime]\cap \{v_1,\ldots,v_n\}\bigr|(H+\gamma/2)}\}\\
&<&\hat{R}^Ne^{|[\XX]\cap[\XX^\prime]|(H+\gamma/2)}.
\end{eqnarray*}

Therefore,

\begin{eqnarray}\label{eq:hefresh}
\E_{\qx}(g_n\cdot{\bf 1}_{A_\XX})
  &\leq& \\ \nonumber
Q_\XX(A_\XX)\left[
\sum_{j=0}^N
\hat{R}^j
\Psi(|[\XX]\cap[\XX^\prime]|=j)
\right.
&+&
\left.
\sum_{j=N}^\infty
e^{j(H+\gamma/2)}
\Psi(|[\XX]\cap[\XX^\prime]|=j)
\right].
\end{eqnarray}

For $j>N+1$,
\begin{equation}\label{eq:itrbnd}
Q_\XX(A_\XX)\Psi\left(\left|[\XX]\cap[\XX^\prime]\right|=j\right)\leq e^{-j(H+\gamma)}.
\end{equation}

Note that in (\ref{eq:hefresh}) and (\ref{eq:itrbnd}) $\XX$ is fixed and the $\Psi$-distributed variable is
$\XX^\prime$.\\
(\ref{eq:condtopr}) follows from \ref{eq:hefresh} and \ref{eq:itrbnd}, and thus we get (\ref{eq:limfi1}).
To see (\ref{eq:limfi2}), we first note that by (\ref{eq:limfi1}),
\[
P(\lim_{n\to\infty}f_n<\infty)>0.
\]
Indeed, by \eqref{eq:limfi1}, $Q$ is absolutely continuous w.r.t. $P$. Therefore, the integral of $\frac{dQ}{dP}$ w.r.t. P is 1, and therefore cannot be $P$-a.s. zero. Therefore $\lim_{n\to\infty}f_n$ cannot be a.s. infinite.
The event $\{\lim_{n\to\infty}f_n<\infty\}$ is a tail event on the i.i.d. distribution $P$, so (\ref{eq:limfi2}) follows from the 0-1 law.
\end{proof}

Now we are able to prove Theorems \ref{thm:branum} and \ref{thm:sinpath}
\begin{proof}[Proof of Theorem \ref{thm:branum}]
Part \ref{item:super} follows immediately from  Definition
\ref{def:bn} and Lemma \ref{lem:super}. Part \ref{item:sub}
follows from Lemma \ref{lem:cutbran} and Lemma \ref{lem:sub}, by
taking the sets $V_n$ to be a sequence of cuts as in Lemma
\ref{lem:cutbran}.

%
\end{proof}

\begin{proof}[Proof of Theorem \ref{thm:sinpath}]
Part \ref{item:sinsub}: For $\gamma>0$, Let $\Upsilon_+^{(\gamma)}=\{\XX\in\partial T: d_\Psi(\XX)>H+\gamma\}$, and similarly we define $\Psi_+^{(\gamma)}$ to be $\Psi$ conditioned on $\Upsilon_+^{(\gamma)}$. Provided that $\Psi\left(\Upsilon_+^{(\gamma)}\right)>0$, for $\Psi_+^{(\gamma)}$ almost every $\XX=(w_1,w_2,\ldots)$,
\[
\lim_{n\to\infty}\Psi_+^{(\gamma)}(w_n)e^{n(H+\gamma)}=0.
\]
Therefore, by Lemma \ref{lem:super}, $Q_{\Psi_+^{(\gamma)}}$ and $P$ are indistinguishable.
\ignore{
For every $A$,
\begin{equation}\label{eq:lima}
Q_+(A)=\lim_{H^\prime\to H}Q_{\Psi_+^{(H^\prime)}}(A)
\end{equation}
and
\begin{equation}\label{eq:ratio}
Q_+(A)\geq Q_{\Psi_+^{(H^\prime)}}(A)
\frac{\Psi(\Upsilon_+^{(H^\prime)}}
{\Psi(\Upsilon_+)}
\end{equation}
Therefore for every $A\subseteq\partial T$, if $P(A)>0$ then there exists some $H^\prime>H$ with $\Psi\left(\Upsilon_+^{(H^\prime)}\right)>0$ and $Q_{\Psi_+^{(H^\prime)}}(A)>0$, so by (\ref{eq:ratio}), $Q_+(A)>0$. If, on the other hand, $P(A)=0$ then $Q_{\Psi_+^{(H^\prime)}}(A)=0$ and by (\ref{eq:lima}), $Q_+(A)=0$.
 }
 As $\gamma\to 0$, the events $\Upsilon_+^{(\gamma)}$ increase and tend to $\Upsilon_+$. Thus by continuity, $P$ and $Q_+$ are indistinguishable.

\vspace{0.3cm}

\noindent Part \ref{item:sinsuper}: For $\gamma>0$,  Let
$\Upsilon_-^{(\gamma)}=\{\XX\in\partial T: d_\Psi(\XX)<H-\gamma\}$,
and similarly we define $\Psi_-^{(\gamma)}$ to be $\Psi$ conditioned
on $\Upsilon_-^{(\gamma)}$.
We choose $\gamma$ so that the probability of $\Upsilon_-^{(\gamma)}$ is positive.
 For every
$\XX=(\XX_0=v,\XX_1,\XX_2,\ldots)\in\Upsilon_-^{(\gamma)}$, we
define $n_0(\XX)=0$ and for every $k\geq 1$ we define $n_k(\XX)$ to
be
\[
n_k(\XX)=\min\left(n>n_{k-1}(\XX):
\frac{-\log(\Psi(\XX_n))}{n}<H-\gamma \right).
\]
We define $V_k=\{\XX_{n_k(\XX)}:\XX\in\Upsilon_-^{(\gamma)}\}$. It is easy to notice that $V_k$ is an antichain.

Clearly, $\Psi_-^{(\gamma)}(V_n)=1$. In addition, every vertex in $V_n$ is at distance at least $n$ from the root. We also know that
$\Psi(u)>e^{-|u|(H-\gamma)}$  for every $u\in V_n$. Therefore, for
\[
C=\Psi\bigl(\Upsilon_-^{(\gamma)}\bigr)^{-1}<\infty
\]
 we get that
 $C\Psi_-^{(\gamma)}(u)>e^{-|u|(H-\gamma)}$. Since
\[
\sum_{u\in V_n}\Psi_-^{(\gamma)}(u)=1,
\]
we see that
\[
\sum_{u\in V_n}e^{-|u|(H-\gamma)}<C,
\]
and remembering that $|u|\geq n$ for every $u\in V_n$, we get that
\[
\sum_{u\in V_n}e^{-|u|H}\leq e^{-n\gamma}\sum_{u\in
V_n}e^{-|u|(H-\gamma)}<Ce^{-n\gamma}
\mathop{\longrightarrow}_{n\to\infty}0,
\]

so $P$ and $Q_{\Psi_-^{(\gamma)}}$ are singular. Again,
continuity finishes the proof.

\end{proof}

\section*{Acknowledgments}
We thank Gidi~Amir, Itai~Benjamini, Nina~Gantert,
Ori~Gurel-Gurevich, Elchanan~Mossel and Gabor~Pete for useful discussions.
We are indebted to Ron Peled for a careful reading of the manuscript, and many helpful comments. Research of N.~B. was partially supported by ISF grant 708/08 and by ERC StG grant 239990.

\bibliography{detect}
\bibliographystyle{plain}

\end{document}